\documentclass[11pt]{amsart}

\usepackage{textcmds} 

\usepackage{mathtools}
\usepackage{tikz-cd}
\usepackage{url}
\usepackage{graphicx,pinlabel}
\usepackage{parskip}
\usepackage{amsthm,amsfonts,amsmath,amssymb,amscd}
\usepackage{verbatim}
\usepackage{enumerate}
\usepackage{mathtools,bm,graphicx}

\newtheorem{thm}{Theorem}[section]
\newtheorem{lma}[thm]{Lemma}
\newtheorem{cor}[thm]{Corollary}
\newtheorem{prp}[thm]{Proposition}

\renewcommand{\tilde}{\widetilde}

\theoremstyle{remark}
\newtheorem{rmk}[thm]{Remark}

\theoremstyle{definition}
\newtheorem{dfn}[thm]{Definition}
\newtheorem{exm}[thm]{Example}

\newcommand{\Z}{\mathbb{Z}}
\newcommand{\C}{\mathbb{C}}

\newcommand{\OP}{\operatorname}
\newcommand{\Sym}{\OP{Sym}}
\newcommand{\Symp}{\mathrm{Symp}}
\newcommand{\Mod}{\mathrm{Mod}}

\newcommand{\id}{\mathrm{id}}

\newcommand{\inv}{\mathrm{inv}}
\newcommand{\Imin}{I_\mathrm{min}}

\usepackage{xcolor}

\mathchardef\ordinarycolon\mathcode`\:
\mathcode`\:=\string"8000
\begingroup \catcode`\:=\active
  \gdef:{\mathrel{\mathop\ordinarycolon}}
\endgroup

\makeatletter
\newcommand*{\da@rightarrow}{\mathchar"0\hexnumber@\symAMSa 4B }
\newcommand*{\da@leftarrow}{\mathchar"0\hexnumber@\symAMSa 4C }
\newcommand*{\xdashrightarrow}[2][]{%
  \mathrel{%
    \mathpalette{\da@xarrow{#1}{#2}{}\da@rightarrow{\,}{}}{}%
  }%
}
\newcommand{\xdashleftarrow}[2][]{%
  \mathrel{%
    \mathpalette{\da@xarrow{#1}{#2}\da@leftarrow{}{}{\,}}{}%
  }%
}
\newcommand*{\da@xarrow}[7]{%
  \sbox0{$\ifx#7\scriptstyle\scriptscriptstyle\else\scriptstyle\fi#5#1#6\m@th$}%
  \sbox2{$\ifx#7\scriptstyle\scriptscriptstyle\else\scriptstyle\fi#5#2#6\m@th$}%
  \sbox4{$#7\dabar@\m@th$}%
  \dimen@=\wd0 %
  \ifdim\wd2 >\dimen@
    \dimen@=\wd2 %
  \fi
  \count@=2 %
  \def\da@bars{\dabar@\dabar@}%
  \@whiledim\count@\wd4<\dimen@\do{%
    \advance\count@\@ne
    \expandafter\def\expandafter\da@bars\expandafter{%
      \da@bars
      \dabar@ 
    }%
  }%
  \mathrel{#3}%
  \mathrel{%
    \mathop{\da@bars}\limits
    \ifx\\#1\\%
    \else
      _{\copy0}%
    \fi
    \ifx\\#2\\%
    \else
      ^{\copy2}%
    \fi
  }%
  \mathrel{#4}%
}
\makeatother

\newcommand{\bq}{\begin{equation}}
\newcommand{\eq}{\end{equation}}

\usepackage[frame,cmtip,curve,arrow,matrix,line,graph]{xy}

\usepackage[margin=3cm]{geometry}

\newcommand{\Fuk}{\mathcal{F}uk}
\newcommand{\AutFuk}{Aut\Fuk}

\usepackage[backref,colorlinks=true,linkcolor=blue,citecolor=blue,urlcolor=blue,citebordercolor={0 0 1},urlbordercolor={0 0 1},linkbordercolor={0 0 1}]{hyperref} 

\title{On symplectic stabilisations and mapping classes}

\begin{document}

\author{Ailsa Keating}

\begin{abstract}

We are interested in comparing properties of symplectic mapping class groups of symplectic manifolds of dimension four or higher with properties of classical mapping class groups of surfaces. 
For $n \geq 2$, consider a configuration of Lagrangian $S^n$s in a Weinstein domain $M^{2n}$. If it is analogous, in some sense that we make precise, to a configuration of exact Lagrangian $S^1$s on a surface $\Sigma$, we show that any relation between Dehn twists in the $S^n$s must also hold between the $S^1$s. Such analogous pairs of configurations include plumbings of $T^\ast S^1$s and $T^\ast S^n$s with the same plumbing graph, and vanishing cycles for a two-variable singularity and for its stabilisation. 
We give a number of corollaries for subgroups of symplectic mapping class groups. 
\end{abstract}

\maketitle


\section{Introduction and statement of the main theorem}

Given an $A_n$ chain of Lagrangian spheres in a Liouville domain $M$, the associated Dehn twists generate a braid group in the symplectic mapping class group of $W$, $\pi_0 \Symp^c(M)$ \cite{Seidel-Thomas, Khovanov-Seidel}. This generalises the classical story for braid groups generated by Dehn twists on Riemann surfaces, including (real) two-dimensional Liouville domains. We also know that a pair of Dehn twists on a Liouville domain of arbitrary  high dimension generate a free subgroup of its symplectic mapping class group under analogous conditions to the two-dimensional case \cite{Keating}.

These motivate the following question: how do properties of symplectic mapping class groups of symplectic manifolds of dimension at least four compare with those of the `classical' two-dimensional mapping class groups? In the present paper, to make this precise,  we focus on cases where there is a clear basis for comparison: some pairs of symplectic manifolds $(M^{2n}, \Sigma^2)$ for which there are configurations of Lagrangian spheres in $M$ and $\Sigma$, say, respectively, $V_i$ and $v_i$,  $i=1,\ldots, k$, with analogous intersection patterns between the $V_i$ and the $v_i$, in a sense that will be defined below. We want to compare relations between the Dehn twists $\tau_{V_i}$, in the symplectic mapping class group $\pi_0 \Symp^c (M)$, and the Dehn twists $\tau_{v_i}$ in $\pi_0 \Symp^c(\Sigma) = \Mod(\Sigma, \partial)$, the mapping class group of $\Sigma$. We also restrict ourselves to Liouville domains (in fact, Weinstein domains), for which Floer and Fukaya-theoretic tools are much further developed.


What are pairs $(M, \Sigma)$ with analogous configurations of exact Lagrangians spheres? Let's start with two classes of examples.

\begin{exm} \emph{Milnor fibres of stabilisations of two-variable singularities.}

Let $f: \C^2 \to \C$ be a two-variable isolated singularity. Pick a Morsification $\tilde{f}: \C^2 \to \C$, and let $a_1, \ldots, a_\mu$ be the critical values of $\tilde{f}$; let $a$ be a regular value. Given a distinguished collection of 
vanishing paths $\gamma_i$ from $a_i$ to $a$, $i=1, \ldots, \mu$, we get a collection $v_1, \ldots, v_\mu$ of vanishing cycles in the Milnor fibre $M_f$ of $f$, a Weinstein domain; these are exact Lagrangian spheres. 

Now consider the stabilisation of $f$, $$F: \C^{2+k} \to \C \,\, , \,\,F(x, y, z_1, \ldots, z_k) = f(x,y) + z_1^2 + \ldots + z_k^2.$$ 
Then $\tilde{f}(x,y) + z_1^2 + \ldots + z_k^2$ is a Morsification of $F$, with, by construction, critical values $a_i$. Let $V_i$ be the vanishing cycle in the Milnor fibre $M_F$  associated to the path $\gamma_i$. We will want to compare properties of Dehn twists in the $V_i$ and the $v_i$. 
\end{exm}

\begin{exm} \emph{Plumbings of $T^\ast S^n$'s.}

Pick a plumbing $\Sigma$ of copies of $T^\ast S^1$. This is determined by a decorated graph $G$. Vertices correspond to $T^\ast S^1$s, and each edge to a plumbing gluing. 
There are two sets of decorations --  first, for each edge, an orientation: for fixed choices of orientation on the $S^1$s, this records whether the plumbing gluing is $(p_1, q_1) = (-q_2, p_2)$ or $(q_2, -p_2)$, where the zero-section coordinates $q_i$ have positive orientation. 
(Notice that  the choice of orientation of each of the $S^1$s is auxiliary: the resulting plumbing does not depend on these. In particular, changing the orientations of all of the edges coming out of a fixed vertex does not change the resulting symplectic manifold, which means that this data only matters in the presence of cycles in $G$.)  
The second decoration is only needed for plumbings of $T^\ast S^1$s (and not in higher dimensions):  for each vertex,  a cyclic ordering of the edges incident to it; given an orientation of the corresponding $S^1$, this gives the order in which to perform the gluings as one travels along the meridional $S^1$.

 Given such a $G$, together with the first set of decorations (an orientation of each edge), for any fixed choice of $n$, one can also construct instead a plumbing of copies of $T^\ast S^n$, say $M$. (The second decoration is no longer needed: $S^n  \backslash \{pt \sqcup pt \}$ is only disconnected in the case $n=1$.) Such pairs $(M, \Sigma)$ will also fall within our framework.
\end{exm}

We will see that these two examples are both special cases of what we call Lefschetz stabilisations:

\begin{dfn}
Given a Liouville domain $F^{2n}$ and a collection of exact Lagrangian spheres $v_1, \ldots, v_k$ in $F$,
 a \emph{Lefschetz one-stabilisation} of $(F, \{ v_i \})$ 
is a pair $(M^{2n+2}, \{ V_i \} )$ consisting of a Liouville domain $M$ and a collection of Lagrangian spheres $V_i$ in $M$, $i=1, \ldots, k$, such that:
\begin{itemize}
\item $M$ is the total space of a Lefschetz fibration (with corners smoothed) with fibre $F$ and base a complex disc, say $\pi: M \to \mathbb{D}$. 
\item $\pi$ has $2k$ critical points, and distinguished collection of vanishing cycles $v_{\sigma(1)}$, \ldots, $v_{\sigma(k)}$, $v_{\sigma(1)}$, \ldots, $v_{\sigma(k)}$, for some permutation $\sigma$ of $\{ 1, \ldots, k \}$. 
\item $V_{\sigma(i)}$ is the matching cycle corresponding to the matching path between the $i^{th}$ and $(i+k)^{th}$ critical points. 
\end{itemize}
Somewhat abusively, we will also call $(F, M; \{ (v_i, V_i ) \}_{i=1, \ldots, k})$ as above a Lefschetz one-stabilisation.
 \end{dfn}

After deformation, we can arrange for the critical values to be at the $(2k)^{th}$ roots of unity, with vanishing paths the straight line segments to the origin. We will assume thereafter that this is the case. Note that the $V_i$ only intersect at $\pi^{-1}(0) = F$.

An example of a one-stabilisation of a two-dimensional Liouville domain is given in Figure \ref{fig:stab1}.

\begin{figure}[htb]
\begin{center}
\includegraphics[scale=0.38]{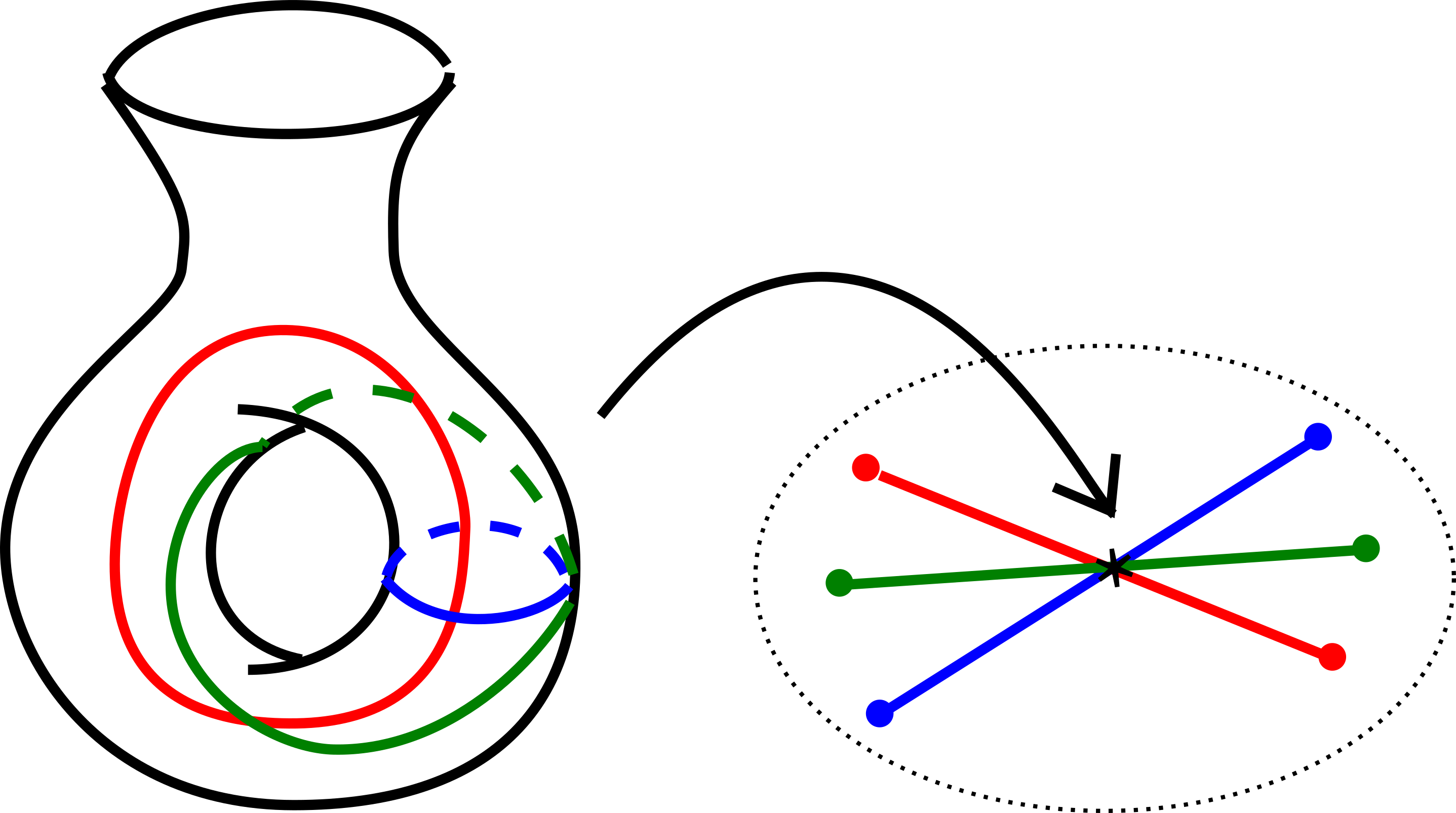}
\caption{Example of a Lefschetz one-stabilisation. The $v_i$ are the curves on the central fibre, and the $V_i$ the corresponding matching cycles.}
\label{fig:stab1}
\end{center}
\end{figure}

\begin{dfn}
Let  $(F^{2l}, \{ v_i \}_{i=1, \ldots, k})$ and $(M^{2n}, \{V_i\}_{i=1, \ldots, k})$
 be Liouville domains with collections of $k$ exact Lagrangian spheres. We say that $(F, M; \{ (v_i, V_i ) \}  )$ is a Lefschetz $(n-l)$--stabilisation (or just a Lefschetz stabilisation) if there is a sequence of Lefschetz one-stabilisations starting with $(F, \{ v_i \})$ and ending with $(M, \{ V_i \})$. 
\end{dfn}

\begin{rmk}
The Lagrangian spheres in each dimension are effectively unordered: in particular, we do not ask that we use the same permutation for two successive Lefschetz fibrations. 
\end{rmk}

\begin{rmk}
We are allowing repeats of the same Lagrangian sphere in the collection $V_i$, although this will not be particularly interesting for us. Of course, the multiplicity of a given object in $\{ v_i \}$ changes the possible Lefschetz stabilisations.
\end{rmk}

Let's check that our two classes of examples do indeed fall within this framework.

\begin{exm} \emph{Milnor fibres of stabilisations of singularities.}

As before, let $f: \C^n \to \C$ be an isolated hypersurface singularity, $\tilde{f}$ a Morsification of it, and $F: \C^{n+1} \to \C$ the stabilisation of $f$: $F( \mathbf{z}, w) = f( \mathbf{z}) + w^2$. Assuming the perturbation $\tilde{f}$ was chosen to be sufficiently small, there is a regular value $a$ of $f$ such that the Milnor fibres $M_f$ and $M_F$ are naturally Liouville submanifolds of $\{ \tilde{f}(\mathbf{z}) = a \}$ and $W:=\{ \tilde{f}(\mathbf{z}) + w^2 = a \}$, respectively. Now the map $\pi: W \to \C$, $(\mathbf{z}, w) \to w$ induces a suitable Lefschetz fibration on $M_F$. 
\end{exm}

\begin{exm} \emph{Plumbings of $T^\ast S^n$s.}

For notational simplicity, let's restrict ourselves to the case of a plumbing along a decorated graph $G$ consisting of a single cycle. 
Let $e_1, \ldots, e_n$ be the vertices of $G$ (ordered by going around the cycle). As each vertex has valency two, there are no choices to be made for the second piece of gluing data (cyclic ordering of edges about each vertex). Pick an arbitrary piece of first data (i.e.~choice of orientation of each edge); by swapping orientations of vertices, we can arrange to have edges arranged positively from $e_1$ to $e_2$, $e_2$ to $e_3$, $\ldots$, $e_{n-1}$ to $e_n$, but have no control on the edge from $e_n$ to $e_1$. 
Let $\sigma = \text{Id}$ if it is oriented positively, and $\sigma = (n-1 , n)$ otherwise. 

Let $\Sigma$ be obtained by plumbing $T^\ast S^1$s according to $G$; call $v_i$ the exact $S^1$ associated to $e_i$. 
Let $W$ be the total space of a Lefschetz fibration over $\C$ with smooth fibre $\Sigma$,  $2n$ critical points, and  distinguished ordered collection of vanishing cycles $v_{\sigma(1)}, \ldots, v_{\sigma(n)}, v_{\sigma(1)}, \ldots, v_{\sigma(n)}$. 
Now notice that up to smoothing corners, $W$ is precisely the plumbing of $T^\ast S^2$s according to $G$. If we iterate, we would get the plumbing of $T^\ast S^3$s along $G$, and so on.

\end{exm}

We are now ready to state our main theorem.

\begin{thm} \label{thm:main}
Let $(\Sigma; \{ v_i \})$ be a real two-dimensional Liouville domain, together with a collection of exact $S^1$s $v_i$. Let $(M^{2n}; \{ V_i \})$ be any Liouville domain and collection of Lagrangian spheres such that $(\Sigma, M^{2n}; \{ (v_i, V_i) \} )$ is an $(n-1)$--stabilisation. Then any relation between the Dehn twists $\tau_{V_i} \in \pi_0 \Symp^c (M)$ must also hold between $\tau_{v_i} \in \pi_0 \Symp^c(\Sigma)$:
$$
\prod_j \tau^{m_j}_{V_{i_j}}=  \id  \in \pi_0 \Symp^c (M) \Rightarrow 
\prod_j \tau^{m_j}_{v_{i_j}}=  \id  \in \pi_0 \Symp^c (\Sigma).
$$
\end{thm}

Up to deformation, any such surface $\Sigma$, and thus by construction $M$, is a (Wein)Stein domain.

In Theorem \ref{thm:main} both $\Symp^c(M)$ and $\Symp^c(\Sigma)$ are equipped with the $C^\infty$ topology, inherited from the compactly supported diffeomorphism groups. In particular, $\pi_0 \Sym^c (\Sigma)$ consists of equivalence classes of compactly supported symplectomorphisms up to symplectic (and not Hamiltonian) isotopy -- and so, by Moser's trick, agrees with the mapping class group $Mod(\Sigma, \partial)$. For $M$, this depends on whether or not $H^1(M)$ vanishes; note however that even if it doesn't, in this particular setting we can always construct Weinstein domains $M'$ with $H^1(M')=0$ and exact open embeddings $M \subset M'$, which we will use.

 Our proof uses Seidel and Smith's work on $\Z/2$--equivariant Floer theory \cite{SeidelSmith}. With care, one would expect to be able to use e.g.~tools from \cite{HLS} to prove that the conclusion of our theorem holds for a broader collection of pairs $(\Sigma, M)$; we have not pursued this here. 

There is a variant of our main theorem at the level of the group quasi-isomorphisms of the Fukaya category of compact Lagrangians, Theorem \ref{thm:main_noncompact}. 
These theorems allow us to `lift' various results about mapping class groups to the higher-dimensional symplectic setting; these corollaries are collected in Section \ref{sec:corollaries}. We record one here:

\begin{cor}(Theorem \ref{thm:virtual_RAAGs})
Fix a group $A$ that is virtually special in the sense of Haglund and Wise, for instance, the fundamental group of any hyperbolic 3-manifold. Then in each dimension greater than two there exist infinitely many simply connected Weinstein domains $M$ such that $A$ embeds into the group of quasi-isomorphisms of $\Fuk(M)$, the Fukaya category of $M$.
\end{cor}


\subsection{Two cautionary examples}

The converse to Theorem \ref{thm:main} is known to be false: relations in the $\tau_{v_i}$ need not hold in the $\tau_{V_i}$. Here are two examples where this fails. 

\subsubsection{Relations in $E_6$ configurations}

Consider the Milnor fibre of the $E_6$ singularity, $\Sigma =  \{x^3+y^4=1\}$. 
Wajnryb \cite{Wajnryb} proved that the Artin group (i.e.~generalised braid group) of type $E_6$ does not embed into the mapping class group of  $\Sigma$. On the other hand, Seidel (\cite[Remark 20.7]{Seidel_book} and \cite[Corollary 6.5]{Seidel_suspending}) shows that Wajnryb's relation does \emph{not} hold for $E_6$ Milnor fibres of sufficiently high dimension: the varieties $\{ z_0^3 + z_1^4 + z^2_2 + \ldots + z_n^2 =1\}$ for $ n \geq  3$. Moreover, Qiu and Woolf \cite{Qiu-Woolf} show that in fact the $E_6$ Artin group acts freely on the Fukaya category of such Milnor fibres; in particular, it embeds into their symplectic mapping class groups.

\subsubsection{The Labru\`ere relation in a four-valent plumbing}\label{sec:Labruere}

Consider a graph $G$ with five vertices: a single central four-valent one and four leaves, such that the associated plumbing of $T^\ast S^1$s is given on the left of Figure \ref{fig:Labruere}. 
The Artin-Tits group associated to this graph, say $H$, has a generator for each of the vertices, say $\sigma_a,  \ldots, \sigma_e$, with the following relations: two generators commute if there is no edge between the corresponding vertices, and have a braid relation if there is. 
In \cite{Labruere}, Labru\`ere showed that the natural map from $H$ to the mapping class group of $\Sigma$, given by mapping $\sigma_a$ to the Dehn twist $\tau_a$, etc., has non-trivial kernel. In the case at hand, the relation in \cite[Section 2.2]{Labruere} boils down to the fact that the Dehn twists in $\tau_a \tau_c d$ and $\tau_b \tau_c e$ commute, whereas their natural preimages in $H$ do not. (Many thanks to Jonny Evans for spotting this version of the relation, which is simpler than the more general one described in the article.) The corresponding two curves are given on the right-hand side of Figure \ref{fig:Labruere}. 

\begin{figure}[htb]
\begin{center}
\includegraphics[scale=0.37]{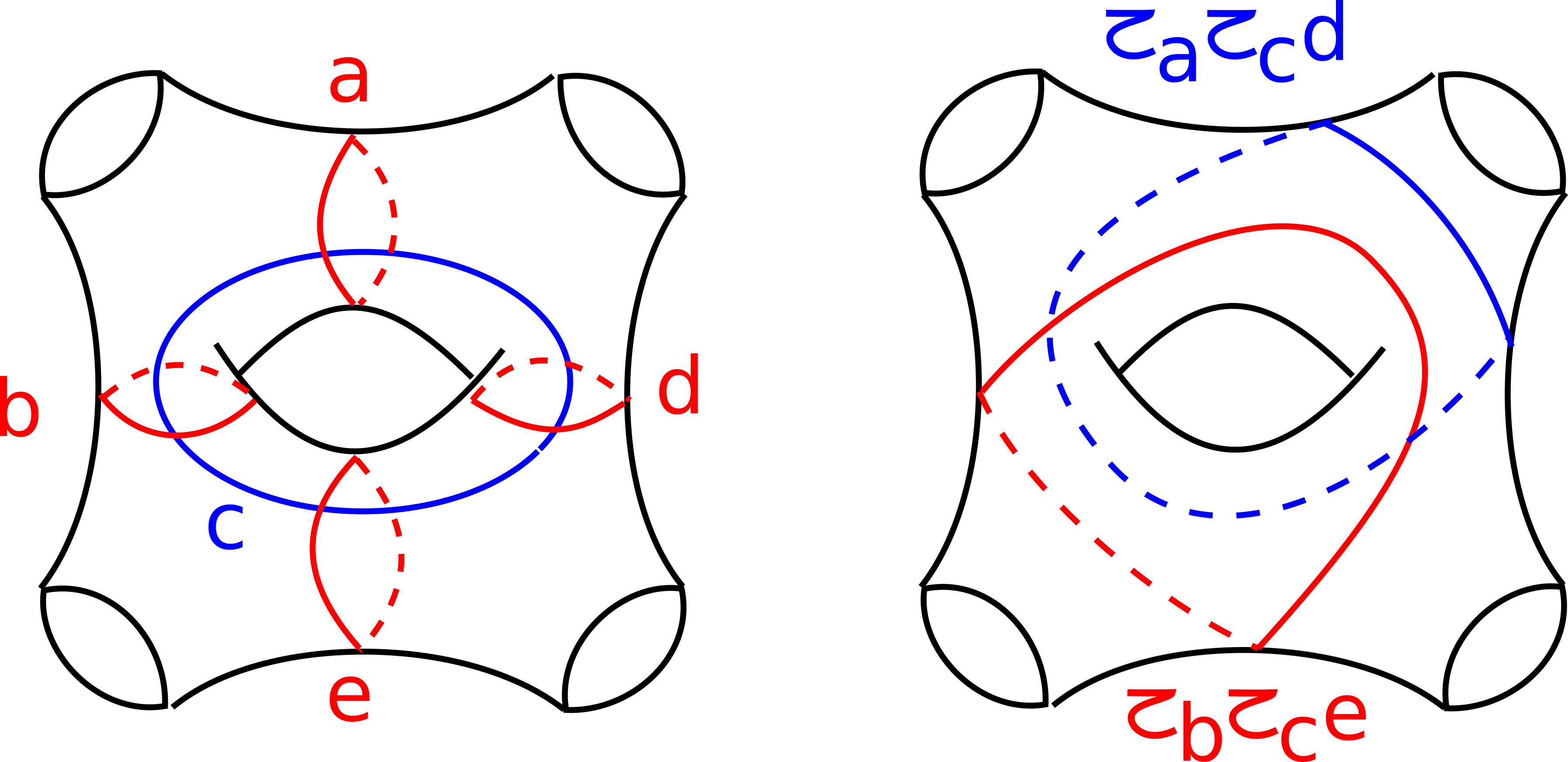}
\caption{Curve configurations for Labru\`ere's relation.}
\label{fig:Labruere}
\end{center}
\end{figure} 

This relation fails to hold in higher dimensions for analogous reasons to the $E_6$ case: consider the corresponding plumbing of $T^\ast S^n$s, for some $n \geq 3$, say $M$. By \cite[Corollary 6.5]{Seidel_suspending}, the subcategory of the Fukaya category of $M$ generated by the five zero-sections is formal; label the corresponding objects by $A, B, C, D$ and $E$. 
We have that 
$$
\xymatrix{
\tau_A \tau_C D \cong \big\{ \, p_D^{\vee}  \otimes p_A \otimes A \ar[r]^-{1 \otimes ev} &
p_D^{\vee} \otimes C \ar[r]^-{ev} & D  \big\} 
}
$$
and 
$$
\xymatrix{
\tau_B \tau_C E \cong \big\{ \, p_E^{\vee} \otimes p_B \otimes B \ar[r]^-{1 \otimes ev} &
p_E^{\vee} \otimes C \ar[r]^-{ev} & E \big\}
}
$$
where $p_A \in CF(A,C)$ is the (unique) generator of the Floer chain complex $CF(A,C)$, corresponding to the transverse intersection point between $A$ and $C$,  $p_A^{\vee} \in CF(C, A)$ is its dual, and similarly for $p_B$, $p_D$ and $p_E$. 

Using formality, one can then calculate that 
$$
dim \, HF (\tau_A \tau_C D, \tau_B \tau_C E) = 2
$$
and so the Dehn twists in $\tau_A \tau_C D$ and $\tau_B \tau_C E$ generate a free subgroup of automorphisms of the Fukaya category instead of commuting.

\subsection*{Conventions}

Given two compact Lagrangians $L_0, L_1$ in a Liouville domain $M$, $HF(L_0, L_1)$ will denote the Floer cohomology group between them with $\Z/2$ coefficients and no gradings. The Fukaya category $\Fuk(M)$ will be defined as in \cite[Section 9]{Seidel_book}: objects associated to compact exact Lagrangians, $\Z/2$ coefficients, no gradings. We'll denote by $\AutFuk(M)$ the group of quasi-isomorphisms of the associated category $\text{Tw} \Fuk (M)$, also definted as in \cite{Seidel_book}.

\subsection*{Acknowledgements} 
I am very grateful to Jonny Evans for numerous discussions regarding higher-dimensional symplectic mapping class groups, encouragements, and feedback on an early version of the draft. In particular, much of Section \ref{sec:Labruere} and Remark \ref{rmk:Mess} stems from conversations with him.

Many thanks also to Henry Wilton for bringing Bridson's work \cite{Bridson} to my attention, and for explanations regarding wreath products. I am also grateful to both him and Ivan Smith for comments on an earlier version of this article.

I was partially supported by NSF grant DMS--1505798, and by NSF grant DMS--1128155 whilst at the Institute for Advanced Study. Thanks to the Institute for a very enjoyable semester, and to Helmut Hofer for his role in making it happen.

\section{Proof of the main theorem}

\subsection{Lagrangian arcs}

\begin{dfn}
An arc on a real two-dimensional Liouville domain $\Sigma$ is the image of an embedding 
$([0,1], \partial) \to (\Sigma, \partial)$ such that $(0,1)$ has image in the interior of $\Sigma$.
A \emph{Lagrangian arc} in  $\Sigma$ is an arc that is invariant under the Liouville flow in a small collar neighbourhood of $\partial \Sigma$.
\end{dfn}

\begin{dfn}
Assume that $c_1$ and $c_2$ are arcs in $\Sigma$ with disjoint boundaries, or embedded $S^1$'s. The \emph{minimal intersection number} of $c_1$ and $c_2$, say $\Imin(c_1, c_2)$, is the minimum of the unsigned intersection numbers between representatives of the isotopy classes rel boundary of $c_1$ and $c_2$. 
\end{dfn}

\begin{lma}\label{lma:arc_detection}

Let $a \subset \Sigma$ be a Lagrangian arc on $\Sigma$. Then there exists a Lagrangian arc $c \subset \Sigma$ with $\Imin (a,c) = 1$. Moreover, given any Lagrangian arc $b$ disjoint from $a$ in a neighbourhood of one of its boundary points,  $c$ can also be arranged to be disjoint from $b$.
\end{lma}

\begin{proof}
Note that any embedded curve $([0,1], \partial) \to (\Sigma, \partial)$ can be isotoped rel boundary to a Lagrangian arc: the curve just needs to be rectified in a collar neighbourhood of the boundary. 

Consider one of the boundary points of $a$. Let $\nu$ be a collar neighbourhood of the boundary component it belongs to. Then we can choose a small Lagrangian arc $c$, contained in $\nu$, that intersects $a$ transversally in a single point. Moreover, given another Lagrangian arc $b$ disjoint from $a$  in a neighbourhood of one of its boundary points, by choosing $c$ to also lie in a sufficiently small neighbourhood of this boundary point of $a$, we can arrange for $c$ and $b$ to be disjoint.
%
%
\end{proof}

Consider two Lagrangian arcs in $\Sigma$ with disjoint boundaries, say $c_1$ and $c_2$. We will use $HF(c_1, c_2)$ to denote the unwrapped Floer cohomology group of $c_1$ and $c_2$, with $\Z/2$ coefficients and no gradings. (If you attach two one-handles to $\Sigma$, one at the boundary of each of $c_1$ and $c_2$, then there is a natural isomorphism $HF(c_1, c_2) \cong HF(s_1, s_2)$, where $s_i$ is the union of $c_i$ and the core of the corresponding handle.)

We will repeatedly use the following elementary fact, of which we recall a proof. 

\begin{lma}\label{lem:Imin=HF}
Assume $a$ and $b$ are Lagrangian arcs, with disjoint boundaries. Then
$$
\Imin (a, b) = \dim HF (a,b).
$$
\end{lma}


\begin{proof}
Pick an auxiliary Lagrangian arc $c$ that is disjoint from $b$, and intersects $a$ transversally at an interior point. 
Add half-infinite cylindrical ends to $\Sigma$ in the standard way; call the resulting Liouville manifold $\tilde{\Sigma}$; by abuse of notation, we will still denote by $a, b$ and $c$ the obvious completions of the Lagrangian arcs.

First, we check that after a compactly supported \emph{Hamiltonian} isotopy of $\tilde{\Sigma}$, we can arrange for $a$ and $b$ to intersect transversally in $\Imin(a, b)$ points. 
To do this, start with  a smooth one-parameter family $a_t$, $t \in [0,1]$, of arcs such that $a_0=a$, $a_t$ agrees with $a$ outside the interior of $\Sigma$ for all $t$, and $a_1$ intersects $b$ minimally. 
Without loss of generality these can be deformed to be Lagrangian arcs, and in such a way that $c$ intersects each of them transversally in one point. 
At each time $t$, we can deform $a_t$ in a tubular neighbourhood of $c$ by `pushing' it along $c$ (with direction depending on the sign of the flux between $a$ and $a_t$) to cancel out the flux between $a$ and $a_t$ -- call the result $a'_t$; this can be done smoothly in $t$. 
As $c$ is disjoint from $b$, we can arrange for this not to introduce intersection points with $b$. 
Thus $a'_1$ and $b$ intersect in $\Imin(a,b)$ intersection points, and, by construction, there is a compactly supported Hamiltonian isotopy of $\Sigma$ taking $a$ to $a'_1$. 

Now, there can't be any non-constant holomorphic disc between any pair of intersection points in $a'_1 \pitchfork b$, by minimality of the intersection number $I(a'_1, b)$ and \cite[Proposition 3.10]{FLP}.
\end{proof}

\subsection{Symplectic involutions and intersection numbers}

We start by recording some useful features of Lefschetz stabilisations.

Given a Lefschetz one-stabilisation, the involution $z \mapsto -z$ of the base $\C$ extends to an involution of the total space. This preserves the symplectic form $\omega$, and a boundary-convex $\omega$-adapted almost complex structure $J$. We will denote this involution by $\iota$. 

\begin{lma}\label{lem:equal_HF}
Let $(F, M; \{ (v_i, V_i) \} )$ be a Lefschetz one-stabilisation. Then for all $i, j$, we have
$$
\text{dim} \, HF (V_i, V_j) = \text{dim} \, HF (v_i, v_j).
$$
\end{lma}

\begin{proof}
The author learnt this from \cite[Section 18]{Seidel_book}. Intuitively, $V_i$ and $V_j$ have the same intersection points as $v_i$ and $v_j$, all lying on the central fibre $F$, i.e.~the fixed locus of $\iota$. Choosing appropriate Floer data, the open mapping theorem implies that all pseudoholomorphic discs contributing to the differential of the Floer complex $CF_M(V_i, V_j)$ must have imagein $F$. This in turn implies that they precisely agree with the discs contributing to the differential of the Floer complex $CF_\Sigma (v_i, v_j)$. The equality is then immediate.
\end{proof}


\begin{lma}\label{lma:equivariant_DT}
Let $(F, M; \{ (v_i, V_i) \})$ be a Lefschetz one-stabilisation. Then for any $i$, the Dehn twist $\tau_{V_i}$ has a representative such that:
\begin{itemize}
\item $\tau_{V_i}$ is induced by an automorphism of the base $\C$ that fixes the critical values of the fibration set-wise.  In particular, for any matching cycle $S$, $\tau_{V_i}(S)$ is also a matching cycle.

\item $\tau_{V_i}$ commutes with the involution $\iota$; in particular, it fixes the central fibre $F$ set-wise. 

\item When restricted to the fixed locus of $\iota$, i.e.~the central fibre $F$, we have that $\tau_{V_i}|_F = \tau_{v_i}$. 
\end{itemize}
\end{lma}
This is a standard model for a Dehn twist in a matching cycle -- see e.g. \cite[Figure 16.3]{Seidel_book}. (Note $\tau_{V_i}$ has compact support: it isn't strictly the lift of an automorphism of the base $\C$, but rather agrees with that lift outside of a neighbourhood of the horizontal boundary.) For an illustration, see Figure \ref{fig:DT}.

\begin{figure}[htb]
\begin{center}
\includegraphics[scale=0.25]{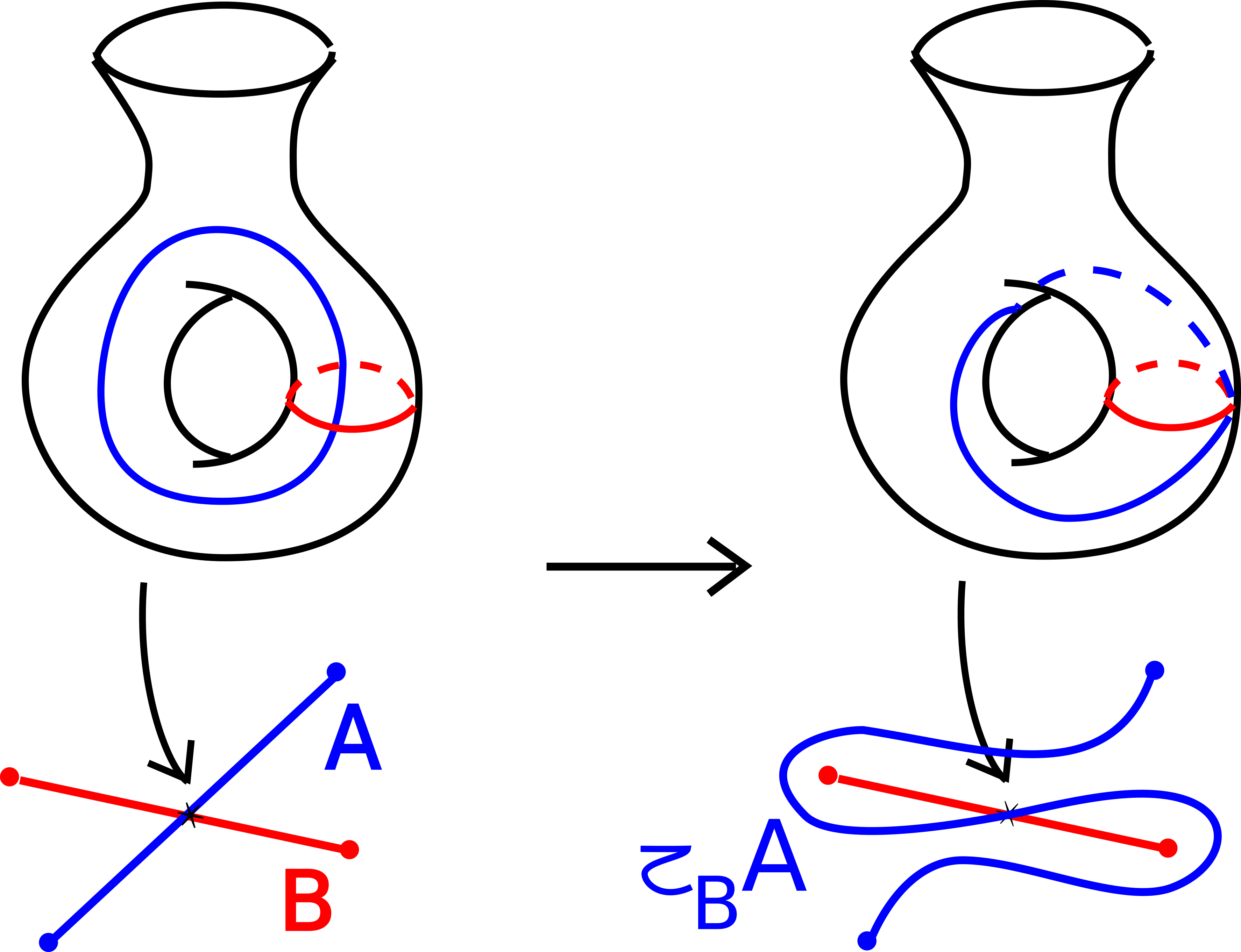}
\caption{$\iota$--equivariant model for a Dehn twist.}
\label{fig:DT}
\end{center}
\end{figure}

For a Lefschetz stabilisation of any length, we will hereafter assume that we have picked representatives for Dehn twists in each of the Lagrangian spheres that satisfy the conditions in Lemma \ref{lma:equivariant_DT}.

A key ingredient will be the following consequence of Seidel-Smith's work  \cite{SeidelSmith}:


\begin{thm} \label{thm:SeidelSmith}  \cite[Theorem 1]{SeidelSmith}.
Fix a Lefschetz $(n-1)$--stabilisation $(\Sigma, M; \{ (v_i, V_i) \} )$
as in the set-up for Theorem \ref{thm:main}. 
Let 
$$
\Phi = \prod_j \tau^{\alpha_j}_{V_{i_j}}
$$
be a word in the Dehn twists in the $V_i$, with $\alpha_j \in \mathbb{Z}$, and 
$$
\phi = \prod_j \tau^{\alpha_j}_{v_{i_j}}
$$
be the same word in the Dehn twists in the $v_i$.
Then, viewing $\Phi$ and $\phi$ as symplectomorphisms of $M$ and $\Sigma$ respectively, there is an inequality of the dimensions of the Floer cohomologies 
$$
\dim HF(\Phi (V_i), V_j) \geq \dim HF( \phi (v_i), v_j)
$$
for all $i, j =1, \ldots, k$. 
\end{thm}

\begin{proof}
Fix $i$ and $j$. 
Consider the symplectic  involution $\iota$ on $M$. By our assumption on Dehn twists, both $\Phi(V_i)$ and $V_j$ are invariant set-wise under $\iota$. Moreover, as they are matching cycles for the Lefschetz fibration on $M$, one readily gets a stably trivial normal structure on $(M, \iota; \Phi(V_i), V_j)$ in the sense of \cite[Definition 18]{SeidelSmith}. 
Thus we can apply \cite[Theorem 1]{SeidelSmith} to the pair $(\Phi(V_i), V_j)$, which gives
$$
\dim HF ( \Phi(V_i) , V_j ) \geq \dim HF ( \Phi(V_i)^\inv , V_j^\inv ) 
$$

Say that $V_i^\inv = U_i$; observe that 
$$
 \Phi(V_i)^\inv = \prod_j \tau^{\alpha_j}_{U_{i_j}}   (U_i).
$$
Now proceed inductively.
\end{proof}

Examining Seidel and Smith's proof \cite[Section 3]{SeidelSmith}, we also see that $\dim HF(\Phi(V_i), V_j)$ and $\dim HF \phi((v_i), v_j)$ have the same parity. Intuitively, this is because the non-invariant generators for $HF(\Phi(V_i), V_j)$ appear in pairs due to the symmetry from the involution. We record the following consequence separately.

\begin{cor} \label{prp:dim_equalities}
Let $(\Sigma, M; \{ (v_i, V_i) \} )$, $\Phi$ and $\phi$ be as in Theorem \ref{thm:SeidelSmith}; suppose moreover that $\dim HF(\Phi(V_i), V_j))$ is equal to zero or one. Then there is an equality
$$
\dim HF ( \Phi(V_i) , V_j ) = \dim HF ( \phi(v_i) , v_j) = 0 \, \, \text{or} \, \, 1.
$$
\end{cor}

We then get the following.

\begin{prp}\label{prp:dim_equalities_id}
Let $(\Sigma^2, M^{2n}; \{ (v_i, V_i) \} )$, $\Phi$ and $\phi$ be as in Theorem \ref{thm:SeidelSmith}, and assume that $\Phi = \id \in \pi_0 \, \Symp^c(M)$.  
Let $c_1$ and $c_2$ be  Lagrangian arcs or exact Lagrangian $S^1$'s. If they are both Lagrangian arcs, we assume that their end-points are distinct. Then we have 
$$
\dim HF (c_1, c_2) \geq \dim HF (\phi(c_1), c_2).
$$
Moreover, equality holds whenever the left-hand side is equal to zero or one.
\end{prp}

\begin{proof}

Let us consider the case where both of the $c_i$'s are Lagrangian arcs. 
First, add two one-handles to $\Sigma$, one for the boundary of each of the $c_i$, so that the union of $c_i$ and the core of the corresponding handle is an exact Lagrangian $S^1$, say $s_i$. Call the new Liouville domain $\Sigma'$; note $\phi$ extends to a symplectomorphism of $\Sigma'$. As 
\[
HF (c_1, c_2) \cong HF (s_1, s_2) \quad \text{and} \quad   HF (\phi (c_1), c_2) \cong HF (\phi (s_1), s_2)
\]
it is enough to proved the claimed (in)equalities for the $s_i$ instead.

We can construct an $(n-1)$--stabilisation of $(\Sigma'; \{ v_i, s_i \} )$, say $(M', \{ V_i, S_i \})$ such that $M$ is a Liouville subdomain of $M'$. 
$\Phi$ is symplectically isotopic to the identity as a compactly supported symplectomorphism of $M$, so, a fortiori, as a symplectomorphism of $M'$. 
Notice that all classes in $H_1(M')$ are induced by classes in $H_1(\Sigma)$ which don't get canceled by any handles when stabilising. In particular, by constructing a Lefschetz stabilisation for $\Sigma$ equipped with a larger collection of exact Lagrangian $S^1$s, we can obtain a Weinstein domain $M''$ with $H^1(M'')=0$, and an open exact embedding $M' \subset M''$. (In fact, we can arrange for $M''$ to be simply connected.) $\Phi$ is symplectically isotopic to the identity as a compactly supported symplectomorphism of $M''$, and so (in $M''$) it must also be Hamiltonian isotopic to the identity.

Now notice that
\begin{multline*}
\dim HF_{M'} (\Phi (S_1), S_2) = \dim HF_{M''} (\Phi (S_1), S_2) = \dim HF_{M''} (S_1, S_2)  \\ = \dim HF_{M'} (S_1, S_2) = \dim HF (s_1, s_2)
\end{multline*}
where for the final equality we use Lemma \ref{lem:equal_HF}. Now the claimed (in)equalities follow from Theorem \ref{thm:SeidelSmith} and Corollary \ref{prp:dim_equalities}.

If  one  (respectively none) of the $c_i$'s is a Lagrangian arc instead, we just make one (respectively zero) handle attachment. 
\end{proof}

\begin{prp}\label{prp:fix_arc}
Let $(\Sigma, M; \{ (v_i, V_i) \} )$, $\Phi$ and $\phi$ be as in Theorem \ref{thm:SeidelSmith}, and assume that $\Phi = \id \in \pi_0 \Symp^c(M)$.  Let $c$ be a Lagrangian arc on $\Sigma$. Then $\phi$ fixes $c$ up to isotopy rel boundary.
\end{prp}

\begin{proof}
 Let $c_1$ and $c_2$ be `parallel' Lagrangian arcs, disjoint from $c$, such that $c_1 \cup c_2$ is the boundary of a tubular neighbourhood of $c$ in $\Sigma$, say $\nu$.
By Proposition \ref{prp:dim_equalities_id}, we have
$$
\Imin ( \phi(c), c_i) = \dim HF (\phi(c), c_i) = 0.
$$
For each $i$, this implies that after isotopy rel boundary, $\phi(c)$ is disjoint from $c_1$ and $c_2$. As $\phi(c)$ has the same boundary points as $c$, it follows that $\phi(c)$ must be contained in $\nu$, where there is a unique Lagrangian arc with those boundary points. 
\end{proof}
%


Our main theorem now follows:

\begin{proof}[Proof of Theorem \ref{thm:main}]
Let $(\Sigma, M^{2n}; \{ (v_i, V_i) \} )$, $\Phi$ and $\phi$ be as in Theorem \ref{thm:SeidelSmith}, and assume that $\Phi = \id \in \pi_0 \, \Symp^c(M)$. We want to show that $\phi = \id \in \pi_0 \, \Symp^c (\Sigma)$. As we are in real dimension two, it's enough to show that $\phi$ is smoothly isotopic (rel boundary) to the identity. We can pick a collection of pairwise disjoint embedded arcs such that their complement in $\Sigma$ is a finite union of discs. By Proposition \ref{prp:fix_arc}, any embedded arc in $\Sigma$ is fixed up to isotopy rel boundary, so we are done.
\end{proof}

\begin{rmk}
We have seen in our proof that for the conclusion of Theorem \ref{thm:main} to hold, it's enough for $\prod \tau^{m_j}_{V_{i_j}}$ to be Hamiltonian isotopic to the identity in some Weinstein domain $M'$ such that there's an exact open embedding $M \subset M'$.
\end{rmk}

\subsection{Restricting to compact objects}

The hypothesis of Theorem \ref{thm:main} is quite strong: $\Phi$ has to be symplectically isotopic, through compactly supported maps, to the identity. In plenty of settings, one could be given an a priori weaker hypothesis -- for instance, asking that $\Phi$ acts as the identity on a flavour of the Fukaya category. 

We expect that if $\Phi$ acts as the identity on a (suitably refined version of) the wrapped Fukaya category of $M$, then the conclusion of our theorem still holds -- but the technical framework to make this rigorous (e.g. a suitable follow-up work to \cite{GPS}) is not yet in place. Instead, we focus on the flavour of the Fukaya category with the least information, $\Fuk(M)$, as in \cite[Section 9]{Seidel_book}.

As we are only allowed to use information about compact objects, we need some further preliminaries. 

\begin{lma}
Let $\Sigma$ be an exact surface, and $v_1, \ldots, v_k$ a collection of embedded simple closed curves on $\Sigma$.
 Then there exists an exact surface $\tilde{\Sigma}$, an exact embedding $\Sigma \subset \tilde{\Sigma}$,
and a collection of exact embedded $S^1$s in $\tilde{\Sigma}$, say $w_1, \ldots, w_l$, such that for any product $\sigma = \prod \tau^{\pm 1}_{v_{i_j}}$ of Dehn twists in the $v_i$, if $\sigma(w_i)$ is isotopic to $w_i$ for each $i$, then $\sigma$ is the identity in $\pi_0 \, \Symp^c (\Sigma)$. 
\end{lma}

\begin{proof}
It's enough for $\sigma$ to fix a finite collection of Lagrangian arcs in $\Sigma$; take $\tilde{\Sigma}$ to be the surface obtained by attaching one handles for each of these, and the $w_i$ to be the resulting collection of exact $S^1$s. 
\end{proof}

We note the further, more technical statement:

\begin{lma}\label{lma:non_cpt}
Let $\Sigma$ be an exact surface, and $v_1, \ldots, v_k$ a collection of embedded exact simple closed curves on $\Sigma$.
 Then there exists an exact surface $\tilde{\Sigma}$ together with an exact embedding $\Sigma \subset \tilde{\Sigma}$, a collection of exact embedded $S^1$s in $\tilde{\Sigma}$, say $w_1, \ldots, w_m$, and a subset $I \subset \{ 1, \ldots, m \}^2$ of pairs of indices  satisfying the following: 
\begin{itemize}
\item For any $(i,j) \in I$, we have that $i \neq j$, and $I_\text{min}(w_i, w_j)$ is 0 or $1$.
\item Suppose that $\sigma = \prod \tau^{\alpha_j}_{v_{i_j}}$ is a product of Dehn twists in the $v_i$ such that for all $(i,j) \in I$, we have $I_\text{min} (\sigma(w_i), w_j) = I_\text{min}(w_i, w_j)$.  Then $\sigma$ is the identity in $\pi_0 \, \Symp^c (\Sigma)$. 
\end{itemize}

\end{lma}

\begin{proof}
This essentially follows from the proof of Proposition \ref{prp:fix_arc}: first pick a collection of arcs $a_1, \ldots, a_l$ such that if they are all fixed by $\sigma$ (relative to their boundaries), then $\sigma$ is the identity in $\pi_0 \Symp^c \, (\Sigma)$. Then pick a further collection of arcs $b_{i,j}$, $i=1, \ldots, l$, $j=1,2,3$ as in the proof of Proposition \ref{prp:fix_arc}: $b_{i,1}$ and $b_{i,2}$ parallel to $a_i$, and $b_{i,3}$ intersecting $a_i$ transversally in a point and living in a collar neighbourhood of one of the boundary components. Now take the surface $\tilde{\Sigma}$ given by attaching a one-handle for each of the arcs $a_i$ or $b_{i,j}$, label the associated exact Lagrangian $S^1$s in $\tilde{\Sigma}$ by $w_1, \ldots, w_m$ for some choice of index ordering, and take the set $I$ to correspond to the collection $\{ (i, (i,j) ) \}$ for all $i=1, \ldots, l$, $j=1,2,3$. 
\end{proof}

\begin{rmk}
We have proceeded quite greedily in our proof: the $\tilde{\Sigma}$ that one obtains this way isn't close to being minimal in general. 
In particular, note that one expects analogues of Lemmas \ref{lma:arc_detection} and \ref{lem:Imin=HF}, and Proposition \ref{prp:fix_arc} for exact Lagrangian $S^1$s, which would further cut down on  the number of handle attachments one needs to make. 

\end{rmk}

Given an exact Lagrangian sphere $V$ in $M$, the Dehn twist $\tau_{V}$ induces an automorphism of $\Fuk(M)$, defined up to quasi-isomorphism, and so an element of $\AutFuk(M)$. 

We get the following variation on Theorem \ref{thm:main}.

\begin{thm}\label{thm:main_noncompact}

Let $(\Sigma; \{ v_i \})$ be a real two-dimensional Liouville domain, together with a collection of exact $S^1$s $v_i$, $i=1, \ldots, k$. Then there exists another real two-dimensional Liouville domain $\tilde{\Sigma}$, together with an exact embedding $\Sigma \subset \tilde{\Sigma}$, and exact $S^1$s $v_i$, $i=k+1, \ldots, l+k$ on $\tilde{\Sigma}$ such that the following holds: 
let $(\tilde{M}^{2n}; \{ V_i \})$ be any Liouville domain and collection of Lagrangian spheres  such that $(\tilde{\Sigma}, \tilde{M}^{2n}; \{ (v_i, V_i) \}_{i=1, \ldots, l+k} )$ is an $(n-1)$--stabilisation. 
 Assume that there is a categorical relation between the Dehn twists in the $V_i$, for indices $i \in \{ 1, \ldots, k \}$, in the following sense: $\prod_j \tau^{m_j}_{V_{i_j}} = \text{id} \in \AutFuk(\tilde{M})$, some $m_j \in \mathbb{Z}$. 
 Then the same relation must also hold between the $\tau_{v_i}$ in $\pi_0 \, \Symp^c(\Sigma)$. 

\end{thm}

\begin{rmk}
We are \emph{not} assuming that the relation between the $\tau_{V_i}$ holds in the symplectic mapping class group $\pi_0 \Symp^c \, (\tilde{M})$. Roughly speaking, we have traded our compact support assumptions on $M$ for something weaker on the larger space $\tilde{M}$. 
\end{rmk}

\begin{proof}
Pick $\tilde{\Sigma}$ as in Lemma \ref{lma:non_cpt}, and $v_{i+k} = w_i$. Proceeding as before, for all $\alpha, \beta \in \{ k+1, \ldots, l \}$, 
\begin{multline*}
I_\text{min} (v_\alpha, v_\beta) = 
HF(V_\alpha , V_\beta) = HF \Big(\prod_j \tau^{m_j}_{V_{i_j}} V_\alpha, V_\beta \Big)
\\
 \geq 
HF \Big(\prod_j \tau^{m_j}_{v_{i_j}} v_\alpha, v_\beta \Big) 
= I_\text{min} \Big(\prod_j \tau^{m_j}_{v_{i_j}} v_\alpha, v_\beta \Big)
\end{multline*}
with equality whenever the left-hand side is equal to zero or one. Thus the hypotheses of Lemma \ref{lma:non_cpt} are satisfied, and we get that $\prod_j \tau^{m_j}_{v_{i_j}}=  \id  \in \pi_0 \Symp^c (\Sigma)$.
\end{proof}

\section{Some corollaries}\label{sec:corollaries}

\subsection{Free groups and right-angled Artin groups}

Given a Lefschetz stabilisation \\ $(\Sigma, M; \{ (V_i, v_i ) \})$,  if there are no relations between Dehn twists in (some of) the $v_i$, then there certainly can't be any between Dehn twists in the corresponding $V_i$. This allows us to `lift' certain free subgroups of classical mapping class groups to higher dimensions.

\begin{prp}
Suppose that $(\Sigma, M; \{ (v_i, V_i )\} )$ is a Lefschetz stabilisation, with $\Sigma$ of real dimension two. 
Suppose that for some subset $I$ of the indices of the $v_i$, and all $i, j, k \in I$ with $i \neq j \neq k$, 
$$
6 I_\text{min} (v_i, v_k) \leq I_\text{min} (v_i, v_j) I_\text{min} (v_j, v_k).
$$
Then the $\tau_{V_i}$, $i \in I$, generate a free subgroup $\mathbb{F}_{|I|}$ of the symplectic mapping class group of $M$. 
\end{prp}

\begin{proof}
This follows from a result of Hamidi--Tehrani \cite[Theorem 7.2]{Hamidi-Tehrani}, combined with our Theorem \ref{thm:main}. 
\end{proof}

Notice that by e.g.~using generalised plumbing constructions (for instance starting with a generalised graph with finitely many vertices and six edges between each pair of vertices, and performing a plumbing of $T^\ast S^n$s according to that graph), we can construct many configurations satisfying the hypotheses of the above proposition. 
Moreover, for any such surface $\Sigma$, we can find a further collection of exact Lagrangian $S^1$ whose (conjugacy) classes generate $\pi_1 (\Sigma)$. 
In particular, we get the following corollary.

\begin{cor}
For any $k \in \mathbb{N}$, and any $n \geq 2$, there are infinitely many simply connected $2n$-dimensional Weinstein domain $M$ whose symplectic mapping class group contains a free subgroup on $k$ elements, generated by Dehn twists. 
\end{cor}

\begin{rmk} If we allow the subgroups to be generated by powers of Dehn twists instead of Dehn twists, results along these lines were already known by \cite{Khovanov-Seidel}: given an $A_k$-chain of Lagrangian spheres, the corresponding Dehn twists generate as a subgroup of the symplectic mapping class group the braid group on $k+1$ strands. This contains a free subgroup on $k$ elements, generated by the elementary pure braids between the first and $i^{th}$ strands, $i=2, \ldots, k+1$. More generally, notice also that the free group on two generators (and so the pure braid group on three strands) contains as a subgroup a free group on countably many generators. 
\end{rmk} 

Using Theorem \ref{thm:main_noncompact}, we can also get the slightly stronger statement:

\begin{cor}
For any $k \in \mathbb{N}$, and any $n \geq 2$, there are infinitely many simply connected $2n$-dimensional Weinstein domains $M$ such that $\AutFuk(M)$, contains a free subgroup on $k$ elements, generated by Dehn twists.
\end{cor}

\begin{rmk}\label{rmk:Mess} (This is a suggestion of Jonny Evans.)
In a somewhat more restricted setting, we can also use Theorems \ref{thm:main} and \ref{thm:main_noncompact} to lift results of Mess \cite{Mess}: fix a free subgroup of the Torelli group of a genus two surface generated by Dehn twists in finitely many separating curves. Puncture (in multiple points) the surface away from representatives for the curves, so that the resulting punctured surface $\Sigma$ can be equipped  with an exact symplectic form such that each of the curves, say $v_1, \ldots, v_k$, is exact. Then the group generated by $\tau_{v_i}, \ldots, \tau_{v_k}$ is a free subgroup of the mapping class group of $\Sigma$, and we obtain free subgroups of mapping class groups of stabilisations of $(\Sigma; \{ v_i \})$. 
\end{rmk}

How about other groups? Let's first recall some definitions from geometric topology, largely following \cite{Koberda}.

\begin{dfn}
Given a graph $\Gamma$, with vertex set $V = \{ \zeta_i \}$, the right-angled Artin group associated to $\Gamma$ is
$$
A(\Gamma) = \langle \zeta_i \, | \, [\zeta_i, \zeta_j ] \text{ whenever there is an edge in }\Gamma 
\text{ between } \zeta_i \text{ and } \zeta_j \rangle.
$$
The set of classes $\{  \zeta_i \}$ is called a right-angled Artin system for $A(\Gamma)$. 
\end{dfn}

\begin{dfn}
Let $v_1, \ldots, v_k$  be a collection of embedded simple closed curves on a surface $\Sigma$. The coincidence graph of the $v_i$ is a graph with a vertex $a_i$ for each $v_i$, and an edge between $a_i$ and $a_j$ precisely when $I_\text{min} (v_i, v_j) =0$. 
\end{dfn}

We shall use the following result of Koberda:

\begin{thm} \cite[Theorem 1]{Koberda}
Let $v_1, \ldots, v_k$  be a collection of embedded simple closed curves on a surface $\Sigma$. Let $\tau_i = \tau_{v_i}$. Assume that the collection is irredundant: none of the $v_i$ is smoothly isotopic to another one.
Then there exists an $N \in \mathbb{N}$ such that for all $n \geq N$, the set of mapping classes $\{ \tau_1^n, \ldots, \tau_k^n \}$ is a right-angled Artin system for a right-angled Artin subgroup of $Mod(\Sigma, \partial)$. Moreover, this subgroup is associated to the graph given by the coincidence correspondence of the $v_i$. 
\end{thm}

\begin{rmk}
Koberda assumes that $\Sigma = \Sigma_{g,p}$, a genus $g$ surface with up to $p$ punctures; the results holds a fortiori for a surface with boundary.
\end{rmk}

Consider a  Lefschetz stabilisation $( \Sigma, M; \{ (v_i, V_i )\} )$ with $\Sigma$ a two-dimensional Liouville domain.  Suppose that the exact Lagrangians $v_i$ and $v_j$ ($i \neq j$) are not smoothly isotopic, and that $I_\text{min} (v_i, v_j) = 0$. Using similar arguments to the proof of Lemma \ref{lem:Imin=HF}, we can arrange for them not to intersect after a Hamiltonian isotopy. Thus $V_i$ and $V_j$ can also be arranged to be disjoint after Hamiltonian isotopy, and $\tau_{V_i}^n$ and $\tau_{V_j}^n$ commute in $\pi_0 \Symp^c (M)$ for any $n \in \mathbb{N}$. As these are the only relations in Koberda's right-angled Artin groups, we can `lift' Koberda's theorem to higher dimensions.

\begin{prp}\label{prp:embed_RAAG}
Consider a  Lefschetz stabilisation $( \Sigma, M; \{ (v_i, V_i )\} )$ with $\Sigma$ a two-dimensional Liouville domain. Let $I$ be a subset of the indices of the $v_i$, without loss of generality $I = \{ 1, \ldots, l \}$, such that the collection $\{ v_1, \ldots, v_l \}$ is irredundant. 
Then there exists an $N \in \mathbb{N}$ such that for all $n \geq N$, the set of mapping classes $\{ \tau_{V_1}^n, \ldots, \tau_{V_l}^n \}$ is a right-angled Artin system for a right-angled Artin subgroup of $\pi_0 \, \Symp^c (M)$. This Artin group is the one associated to the coincidence graph  of the $v_i$. 

Moreover, we can construct a larger Weinstein domain $\tilde{M}$ and an exact embedding $M \subset \tilde{M}$ such that the same conclusion holds for $\{ \tau^n_{V_1}, \ldots, \tau^n_{V_l} \} $ as a subset of $\AutFuk(\tilde{M})$. 
\end{prp}

By building a surface $\Sigma$ with a suitable collection of exact Lagrangian $S^1$s (for instance, using plumbings), we get the following.

\begin{cor}\label{cor:RAAGs}
Given any right-angled Artin group $A$, and any $m \geq 2$, there exist  infinitely many $2m$--dimensional simply connected Weinstein domains $M$ such that $A$ embeds into $\pi_0 \, \Symp^c (M)$; the generators of $A$ are given by powers of Dehn twists.
Moreover, we can arrange for the embedding to also hold into $\AutFuk(M)$. 
\end{cor}

\begin{rmk}
The generators of $A$ are of the form $\tau^n_{V}$, for any sufficiently large $n$. If $m$, the complex dimension of $M$, is even, then any Dehn twist $\tau_V$ has finite order (up to isotopy) as a compactly supported diffeomorphism \cite{Krylov}. Thus in those cases the RAAG also lies in the kernel of the forgetful map 
$$
\pi_0 \, \Symp^c (M) \to \pi_0 \, \text{Diff}^c (M).
$$
\end{rmk}

\subsection{Decision problems}

We can use Corollary \ref{cor:RAAGs} to apply some decision-theoretic results about RAAGs, due to Bridson, to symplectic mapping class groups. We give basic relevant definitions; for further background, see \cite{Miller}. 

Given a finitely generated group $G$, the conjugacy problem asks for an algorithm that will determine, given a pair of words, whether they are conjugate elements of $G$. The membership problem for a subgroup $H$ of $G$ asks for an algorithm that, given a word in the generators of $G$, will determine whether the corresponding element in $G$ lies in $H$ or not.

\begin{thm}\label{thm:Bridson2} \cite[Theorem 1.2]{Bridson}
There exists a right angled Artin group $A$ and a finitely presented subgroup $H$ of $A$ such that the conjugacy and membership problems are unsolvable for $H$.
\end{thm}

\begin{cor}\label{cor:conjugacy_and_membership}
For any $n \geq 2$, we can construct infinitely many simply connected Weinstein domains $M$ of dimension $2n$ such that there are finitely presented subgroups of $\pi_0 \, \Symp^c (M)$ with unsolvable conjugacy and membership problems. Similarly with $\AutFuk (M)$. 
\end{cor}

Given a collection of finitely presentable groups, the isomorphism problem asks for an algorithm that, given a pair of presentations of groups in the collection, will determine whether the groups are isomorphic.

\begin{thm}\label{thm:Bridson1} \cite[Theorem 1.1]{Bridson}
There exists a right angled Artin group $A$ such that the isomorphism problem for the finitely presented subgroups of $A$ is unsolvable. 
\end{thm}

\begin{cor}\label{cor:isomorphism_pb}
For any $n \geq 2$, we can construct infinitely many simply connected Weinstein domains $M$ of dimension $2n$ such that the isomorphism problem for subgroups of $\pi_0 \, \Symp^c(M)$ is unsolvable. Similarly with $\AutFuk(M)$. 
\end{cor}

The reader might wish to compare this with Seidel's results on undecideability and symplectic cohomology \cite[Corollary 6.8]{Seidel_biased}.

\subsection{Virtually special groups}

A group $H$ virtually embeds into a group $\Gamma$ if $H$ has a finite index normal subgroup $H_0$ such that $H_0$ embeds into $\Gamma$.
Part of the importance of RAAGs in geometric topology comes from the fact that large classes of groups virtually embed into them.
For example, any virtually special group, in the sense of Haglund and Wise \cite{Haglund-Wise}, virtually embeds into a RAAG. Virtually special groups include, for instance, the fundamental group of any hyperbolic three-manifold \cite{Agol, Bergeron-Wise}, and any finitely generated Coxeter group \cite{Haglund-Wise_coxeter}.

In order to make use of this, we will generalise a construction of Bridson \cite[Section 5]{Bridson}. We start with more background from geometric topology.

\begin{dfn}
The wreath product $ \Gamma \wr G$ of groups is the semi-direct product $G \ltimes \prod_{g \in G } \Gamma_g$, where each $\Gamma_g$ is an isomorphic copy of $\Gamma$, and $G$ acts by left translation. 
\end{dfn}

\begin{thm}Kaloujnine-Krasner embedding \cite{KK}.
Suppose $H_0 \lhd H$ is a finite index normal subgroup, with quotient $G$, and that there exists an embedding $H_0 \hookrightarrow \Gamma$ for some group $\Gamma$. Then the natural map $H \to \Gamma \wr G$ is also an embedding. 
\end{thm}

Next, we present a variation of the construction of the proof of \cite[Proposition 5.1]{Bridson}.
Let $\Sigma$ be a real two-dimensional Liouville domain.  Let $G$ be a finite group. This has a realisation as a group of symmetries of a closed surface. Realise $G$ on such a surface, and equivariantly delete an open disc about each point in a free orbit. Let $S$ be the resulting surface with boundary. This can be equipped with the structure of an exact Liouville domain; moreover, by averaging over the action of $G$, we can assume that the Liouville form is $G$--equivariant. Now take $|G|$ copies of $\Sigma$, labeled, say, as $\Sigma_g$, for $g \in G$. Fix a component of $\partial \Sigma$. For each $g \in G$, perform a boundary connected sum between the corresponding component of $\partial \Sigma_g$ and the boundary component of $S$ labeled by $g$. Let $\Sigma_G$ be the resulting surface with boundary; it carries an induced $G$--action, and can be equipped with a $G$--equivariant Liouville form which agrees with the Liouville forms on $S$ and each of the $\Sigma_g$ outside of a neighbourhood of the one-handles used for the boundary connected sum. 

Next, we generalise this construction to higher dimensions, as follows. 

\begin{lma}\label{lma:equivariant_stabilisation}
Let $(\Sigma, \{ v_1, \ldots, v_k \} )$ be an exact real two-dimensional Liouville domain together with a collection of exact Lagrangian $S^1$s. For a finite group $G$, let $\Sigma_G$ be as above, and let $v^g_i$ be the copy of $v_i$ in $\Sigma_g \subset \Sigma_G$. Then in any dimension, there exist Lefschetz stabilisations of $(\Sigma_G, \{ v^g_i  | \, g \in G, i=1, \ldots, k  \} )$ that carry a $G$--action extending that on $\Sigma_G$. 
This action does not have compact support, but can be arranged to be strictly exact. 
\end{lma}

(Recall $f$ is strictly exact if $f^\ast \theta = \theta$, where $\theta$ is the Liouville form.)

\begin{proof}
Enumerate the elements of $G$ as $g_1, \ldots, g_{|G|}$. 
Consider the Lefschetz fibration with fibre $\Sigma_G$ and distinguished collection of vanishing cycles:
$$
v_1^{g_1}, v_1^{g_2}, \ldots, v_1^{g_{|G|}},  v_2^{g_1}, v_2^{g_2}, \ldots, v_k^{g_{|G|}}, v_1^{g_1}, v_1^{g_2}, \ldots, v_1^{g_{|G|}}, 
v_2^{g_1}, v_2^{g_2}, \ldots, v_k^{g_{|G|}}.
$$
Note that for any $i$, the cycles $v_i^{g_1}, \ldots, v_i^{g_{|G|}}$ are pairwise disjoint. In particular, we can deform the Lefschetz fibration until the first $|G|$ critical values merge, the next $|G|$ of them also merge, etc, to get a fibration with $2k$ critical values, and generalised vanishing cycles of the form $v_i^{g_1} \sqcup v_i^{g_2} \sqcup \ldots \sqcup v_i^{g_{|G|}}$. See Figure \ref{fig:Bridson}.

\begin{figure}[htb]
\begin{center}
\includegraphics[scale=0.25]{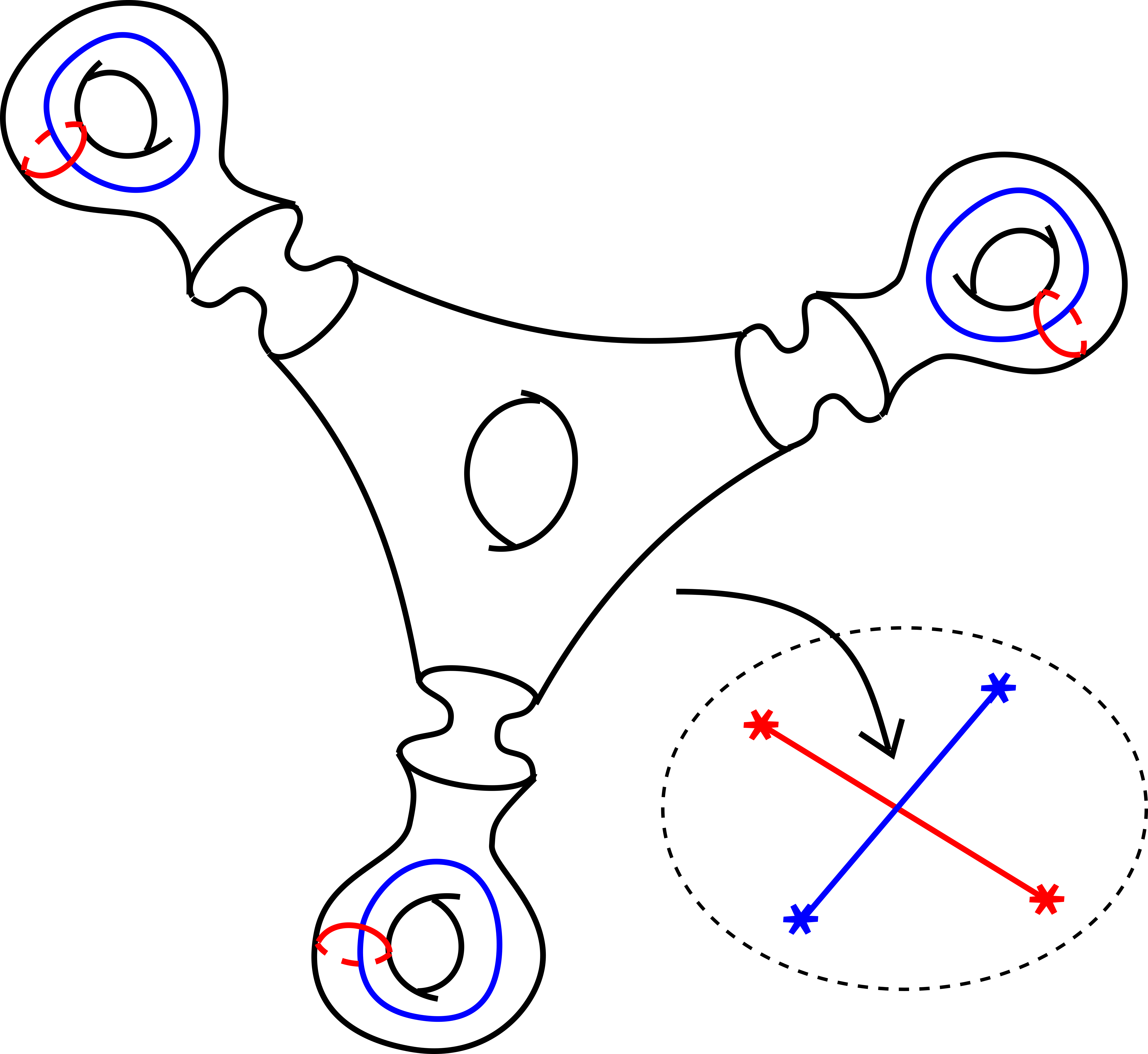}
\caption{Fibration with total space $\Sigma_G$; each of the critical values corresponds here to a disjoint union of three vanishing cycles, and comes from ``pulling together'' three Morse critical points.}
\label{fig:Bridson}
\end{center}
\end{figure}

Now the $G$--action on the central fibre readily extends to the total space, fixing each fibre set-wise. Proceed iteratively to get equivariant Lefschetz stabilisations in arbitrary dimensions.
\end{proof}

\begin{lma}\label{lma:simply_connected}
Consider the Lefschetz stabilisation contructed in Lemma \ref{lma:equivariant_stabilisation}, say $M_G$. There exists a simply connected Weinstein domain $M'_G$ and an exact open embedding $M_G \subset M'_G$ such that the action of $G$ on $M_G$ extends to an action on $M'_G$, also by strictly exact symplectomorphisms.
\end{lma}

\begin{proof}
$M_G$, the total space of the fibration we've constructed in the proof of Lemma \ref{lma:equivariant_stabilisation}, is not simply connected: classes from $\pi_1(S)$, and possibly $\pi_1(\Sigma_g)$, survive. Let $w_1, \ldots, w_h$ be exact embedded $S^1$s in $\Sigma$ whose (conjugacy) classes generate $\pi_1(\Sigma)$, and $w_i^g \subset \Sigma_g$ as before. Similarly, let $s_1^e, \ldots, s_l^e$ be exact embedded $S^1$s in $S$ whose (conjugacy) classes generate $\pi_1(S)$, and set $s_i^g = g(s_i^e)$ for all $g \in G$ ($e$ denotes the identity in $G$). 

Now construct a $G$-equivariant Lefschetz stabilisation of 
$$
\big( \Sigma_G, \{ w_p^g, v_q^g, s_r^g \, | \, p=1, \ldots h, q = 1, \ldots, k, r=1, \ldots l, g \in G \} \big)
$$
as in the proof of Lemma \ref{lma:equivariant_stabilisation}.
\end{proof}

\begin{thm}\label{thm:virtual_RAAGs}
Suppose that some group $H$ virtually embeds into a right angled Artin group $\Gamma$. Then, for any $n \geq 2$, there exist infinitely many simply connected $2n$--dimensional Weinstein domains $M$ such that $H$ embeds into $\AutFuk(M)$. 
\end{thm}

\begin{rmk}
Each element of $H$ can be realised as an exact symplectomophism of $M$ (recall an exact symplectomorphism $h$ satisfies $h^\ast \theta = \theta + df$, some function $f$ with support on the interior of $M$), though in our construction they do not all have compact support. 
\end{rmk}

\begin{proof}
Roughly speaking, this is a higher-dimensional version of Bridson's argument in \cite[Proposition 5.1]{Bridson}. 

Start with an exact surface $\tilde{\Sigma}$ and exact Lagrangians $v_1, \ldots, v_k$ such that for any Lefschetz stabilisation of $(\tilde{\Sigma}, \{ v_1, \ldots, v_k \})$, say $(\tilde{M}, \{ V_1, \ldots, V_k \} )$, the RAAG $\Gamma$ embeds into $\AutFuk (\tilde{M})$, constructed in Proposition \ref{prp:embed_RAAG}.

Say that $H_0 \lhd H$ embeds into $\Gamma$, and let $G = H / H_0$. 
Now consider $(\tilde{\Sigma}_G, \{ v^g_i \} )$ as above, together with its $G$--equivariant Lefschetz stabilisation given by  Lemma \ref{lma:equivariant_stabilisation}, say $(\tilde{M}_G, \{ V^g_i \} )$. For each $g \in G$, the Lefschetz stabilisation of $(\tilde{\Sigma}_g, \{ v_1^{g}, \ldots, v_k^g \} )$, say $(\tilde{M}_g, \{ V_1^{g}, \ldots, V_k^g \} )$
 naturally sits inside  $(\tilde{M}_G, \{ V^g_i \} )$; $\tilde{M}_g$ is a Stein subdomain of $\tilde{M}_G$, and the $\tilde{M}_g$ are disjoint for different $g \in G$. 

For each $g \in G$, let $\Gamma_g$ be an isomorphic copy of $\Gamma$.
We claim that the 
 direct product $\prod_{g \in G} \Gamma_g$ embeds into $\AutFuk (\tilde{M}_G)$. This can be viewed as a special case of Proposition \ref{prp:embed_RAAG}, as $\prod_{g \in G} \Gamma_g$ is itself a RAAG.
 By construction, if that map $\prod_{g \in G} \Gamma_g \to \AutFuk (\tilde{M}_G)$ isn't injective, then the map $\prod_{g \in G} \Gamma_g \to \prod_{g \in G} \pi_0 \, \Symp^c (\Sigma_g)$ isn't either -- however, we know that the latter map \emph{is} injective.

On the other hand, the action of $G$ on $\tilde{M}_G$ 
permutes the $V_i^g$ for fixed value of $i$; thus conjugation by $G$ permutes the $\Gamma_g \subset \AutFuk(\tilde{M}_G)$, and the canonical map $G \ltimes  \prod_{g \in G} \Gamma_g \to \AutFuk(\tilde{M}_G)$ is injective.  

To complete the proof, it suffices to pass to a simply connected Weinstein domain, say $\tilde{M}'_G$ with $\tilde{M}_G \subset \tilde{M}'_G$, as in Lemma \ref{lma:simply_connected}. 
\end{proof}

\begin{rmk}
One could use Theorem \ref{thm:virtual_RAAGs} to import further undecideability results from geometric group theory, see e.g.~\cite[Theorem B]{BridsonWilton}.
\end{rmk}

\bibliography{bib}{}
\bibliographystyle{alpha}

\end{document}